\documentclass{article}

\usepackage[cp1251]{inputenc}            % Выбор языка и кодировки
\usepackage[T2A]{fontenc}
\usepackage{amsmath}
\usepackage{amssymb}
\usepackage{amsthm}

\usepackage[english, russian]{babel}     % Настройки для русского языка

\usepackage{color}
\definecolor{darkgreen}{rgb}{0,.5,0}
\usepackage[colorlinks,
            filecolor=blue,
           linkcolor=blue,
            citecolor=darkgreen]{hyperref}
\numberwithin{equation}{section}
\theoremstyle{plain}
\newtheorem*{theorem*}{Теорема}
\newtheorem*{theoremA}{Теорема Алберта}
\newtheorem{theorem}{Теорема}

\newtheorem{lemma}{Лемма}
\newtheorem{proposition}{Предложение}[section]

\theoremstyle{definition}
\newtheorem{definition}{Определение}
\newtheorem{remark}{Замечание}

\begin{document}

УДК 512.554.5

\textbf{О коммутативных изотопах йордановой алгебры}

\textbf{симметрической билинейной невырожденной }

\textbf{формы на двумерном векторном пространстве}

\textbf{\qquad}

\textbf{В.\,И.~Глизбург, С.\,В.~Пчелинцев}

\begin{abstract}
Ненулевой элемент $x$  алгебры называется \emph{ниль-элементом индекса} $2$, если $x^2=0$.
Назовем \emph{ниль-рангом алгебры} максимальное число линейно независимых ниль-элементов индекса $2$,
содержащихся в ней.
В \cite{KP2020} изучались простые коммутативные $3$-мерные алгебры ниль-ранга $3$.
Известно \cite{KP2020}, что всякая такая алгебра изотопна циклотомичной алгебре $C_3$.
В работе изучаются простые унитальные коммутативные $3$-мерные алгебр над алгебраически замкнутым полем характеристики не 2.
Доказано, что всякая простая унитальная коммутативная $3$-мерные алгебра ниль-ранга $2$
изотопна йордановой алгебре симметрической билинейной невырожденной формы на $2$-мерном векторном пространстве.

Библиография: 13 наименований

Ключевые слова: изотоп, стандартный изотоп, 3-мерная алгебра, коммутативная алгебра, простая алгебра, ниль-ранг.

Bibliography: 13 titles.

Keywords: isotope, standard isotope, 3-dimensional algebra, commutative algebra, simple algebra, nil-rank.

\end{abstract}

\section*{Введение}\label{s1}
Алгебра называется \emph{простой}, если она имеет ненулевое умножение и не содержит собственных идеалов.
Хорошо известно (и легко доказывается индукцией по размерности),
что \emph{в любой конечно-мерной алгебре $A$ существует композиционный ряд
\[A=A_0 \rhd A_1 \rhd A_2 \rhd ...\rhd A_k=0,\]
где $A_{i+1}\lhd A_i$  (идеал) и фактор-алгебры $A_i/A_{i+1}$, $i=\overline{0,k-1}$
либо простые алгебры, либо $1$-мерные алгебры с нулевым умножением.}

Классификация простых конечно-мерных алгебр является важнейшей задачей структурной теории.

В 1942 г. Алберт \cite{Alb1942} ввел понятие изотопа.

Алгебра называется \emph{изотопно-простой}, если всякий ее изотоп является простой алгеброй.
Простая алгебра с единицей является изотопно-простой.

Интерес к изучению изотопно-простых алгебр связан с теоремой Брака
о композиционном ряде \cite{Bru1944}: \emph{всякая конечно-мерная алгебра $A$
обладает изотопом  $B$, в котором существует композиционный ряд
\[B=B_0 \rhd B_1 \rhd B_2 \rhd ...\rhd B_k=0,\]
 где $B_{i+1}\lhd B_i$  и фактор-алгебры $B_i/B_{i+1}$, $i=\overline{0,k-1}$
либо изотопно-просты, либо $1$-мерные алгебры с нулевым умножением.}

В той же работе \cite{Bru1944} Брак доказал, что над алгебраичеки замкнутым полем
не существует $2$-мерных изотопно-простых алгебр.
Известно \cite{KP2020}, что каждая простая $3$-мерная антикоммутативная алгебра изотопна алгебре Ли
$sl_2$ бесследных матриц $2$-го порядка.

Ненулевой элемент $x$  алгебры называется \emph{ниль-элементом индекса} $2$, если $x^2=0$.
Назовем \emph{ниль-рангом алгебры} максимальное число ее линейно независимых ниль-элементов индекса $2$.
В \cite{KP2020} изучались простые коммутативные $3$-мерные алгебры с ниль-базисом (или алгебры ниль-ранга $3$).
Известно \cite{KP2020}, что всякая простая коммутативная $3$-мерная алгебра
с ниль-базисом изотопна циклотомичной (cyclotomic) алгебре  $C_3$.

В работе \cite{KP2020} приведен достаточно полный перечень статей, относящихся к изучению изотопов.
Значительное число статей посвящено изучению изотопов, связанных с автоморфизмами
коммутативных ассоциативных алгебр (см. обзор \cite{Tka2024}).

В статьях \cite{GP2020}--\cite{Pch2021} изучались изотопы
альтернативных, (-1,1) и йордановых алгебр и тождества изотопов.
В \cite{GP2020} было показано, что среди изотопов йордановой алгебры
симметрической билинейной невырожденной формы на бесконечно-мерном
векторном пространстве встречаются как простые алгебры, так и алгебры,
содержащие идеалы с нулевым умножением.

Всюду ниже, термин <<алгебра>> означает линейную алгебру над полем .
Если не оговорено противное, то алгебра предполагается унитальной коммутативной.

В данной работе изучаются простые $3$-мерные алгебры.
Основным результатом данной статьи является классификация (с точностью до изопотии)
простых $3$-мерных алгебр ниль-ранга $2$. Более точно, доказано, что

\emph{всякая простая $3$-мерная унитальная коммутативная алгебра ниль-ранга $2$ 
 над алгебраически замкнутым полем 
характеристики, отличной от $2$, изотопна
 йордановой алгебре $J_2$ симметрической билинейной невырожденной формы на $2$-мерном векторном пространстве}.

Доказано также, что всякая простая коммутативная $3$-мерная (необязательно унитальная) алгебра с ниль-базисом
изотопна унитальной коммутативной алгебре вида
\begin{center}
$C(\rho)=\langle 1,x,y|x^2=y^2=0,\,xy=yx=\rho(1 + x + y)\rangle$ при $\rho=-2$,
\end{center}
где до вертикальной черты записан базис алгебры, 1 -- нейтральный элемент по умножению,
 а после черты указаны произведения базисных элементов.

В связи с изучением алгебр с ниль-базисом введены, так называемые,
простые сильно вырожденные алгебры. Указаны примеры таких алгебр.
Доказано, что каждая такая алгебра не является изотопно-простой.
В частности, показано, что существуют простые алгебры с нетривиальным идемпотентом,
не являющиеся изотопно-простыми.

\section{Изотопы Алберта}\label{s2}
Напомним, что линейные алгебры   $(A,\cdot)$ и $(B,\circ)$
над основным полем $F$ называются \emph{изотопными},
если существует тройка $\Lambda=(\varphi,\psi,\xi)$ линейных изоморфизмов  $A\rightarrow B$,
удовлетворяющих условию
$x^\varphi\circ y^\psi =(x\cdot y)^\xi$   для любых $x,y\in A$.
При этом тройка изоморфизмов $\Lambda$  называется \emph{изотопией}.

Тройка $(\xi,\xi,\xi)$ является изотопией тогда и только тогда, когда $\xi:A\rightarrow B$
является изоморфизмом.

Пусть $f,g: A\rightarrow A$~--- обратимые линейные операторы на алгебре $(A,\cdot)$.
Определим алгебру $A^{\ast}=A^{(f,g)}$  на том же линейном пространстве $A$
относительно нового умножения  $x\ast y =x^f y^g$   для любых $x,y\in A$.
Алгебра $A^{\ast}$  называется \emph{главным изотопом} алгебры $A$.

Известно (см., \cite{Alb1942}, \cite{Pch2020}), что
\emph{две алгебры изотопны тогда и только тогда, когда одна из них изоморфна главному изотопу другой}.
В частности, если алгебры $A$  и $B$  унитальны и изотопны, то главный изотоп $B^{\ast}$  также унитален.

Далее, отношение изотопии является отношением эквивалентности, причем более широким, чем отношение изоморфизма.

\begin{definition}
  Элемент  $a$   алгебры   $A$  называется   $R$\emph{-обратимым},
  если оператор $R_a:x \mapsto xa$  правого умножения на элемент $a$  обратим.
\end{definition}
Заметим, что обратный оператор  $R_a^{-1}$, вообще говоря, не является оператором правого умножения (см. замечание 3).

Всюду ниже, если не оговорено противное, основное поле $F$ предполагается алгебраически замкнутым характеристики не 2.
Через $F^{\times}$  обозначается множество ненулевых скаляров поля  $F$.
\begin{lemma}\label{l:1}
Пусть $A$~--- произвольная алгебра,  $\sigma,\tau \in F^{\times}$ и $1=1_A$~--- тождественное отображение алгебры $A$ на себя.
Тогда изотоп $A^{\ast}= A^{(\sigma 1,\tau 1)}$
и исходная алгебра  $A$ изоморфны. В частности, одномерная алгебра
и любой её главный изотоп изоморфны.
\end{lemma}
\begin{proof}
Пусть  $\omega=(\sigma\tau)^{-1}$. Проверим, что отображение
\[\xi:A\rightarrow A^{\ast},\quad a^{\xi}=\omega a,\]
является изоморфизмом алгебр  $A$ и $A^{\ast}$.
Поскольку гомотетия $\xi$ является линейным изоморфизмом,
то достаточно проверить, что она является мультипликативным гомоморфизмом.
Имеем:
  \[(ab)^{\xi}=\omega ab,\qquad a^{\xi}\ast b^{\xi}=(\omega a)\ast(\omega b)=\omega^2\sigma\tau ab =\omega ab,\]
что и доказывает первую часть леммы.

Пусть ${\rm dim}_F(A)=1$. Тогда любой линейный оператор на $A$ является \emph{скалярным},
т.е. имеем вид $\sigma 1$, где $\sigma \in F$. Умножение в алгебре является нулевым тогда и только тогда,
когда оно нулевое в любом ее изотопе. Значит, справедливо и второе утверждение леммы.
\end{proof}
\begin{theoremA}
Если изотоп $A^{\ast}=A^{(\varphi,\psi)}$  унитален,
то существуют элементы  $g,h\in A$
такие, что $\varphi R_g =\psi L_h = 1_A$, где $L_h$~--- оператор левого умножения на элемент $h$.
\end{theoremA}
\begin{remark}
Всякая одномерная алгебра над алгебраически замкнутым полем $F$
либо имеет нулевое умножение, либо изоморфна полю $F$.
\end{remark}
\begin{remark}
Если изотоп  $A^{(\varphi,\psi)}$  коммутативен,
то для любых $\sigma,\tau \in F^{\times}$
изотоп  $A^{(\sigma\varphi,\tau\psi)}$ также коммутативен.
\end{remark}

\section{Йорданова алгебра $J_2$ симметрической билинейной невырожденной формы}\label{s3}

\begin{definition}
Коммутативная алгебра $J$  называется \emph{йордановой}, если в ней выполнено тождество
\[(x^2,y,x)=0, \]
где $(u,v,w)=(uv)w-u(vw)$~--- ассоциатор элементов $u,v,w$.
\end{definition}
В теории йордановых алгебр важную роль играет простая алгебра $J_n$
симметрической билинейной невырожденной формы,
определенной на $n$-мерном векторном пространстве  (см. \cite{Jac1968}, \cite{ZSSS}).
Напомним её определение.

Пусть $V_n$~--- $n$-мерное векторное пространство над полем  $F$,
снабженное симметричной билинейной формой  $f(x,y)$.
Построим векторное пространство $J_n=F1 \oplus V_n$,
которое является прямой суммой пространства $V_n$ и одномерного пространства $F1$  с базисом $1$.
Определим произведение в $J_n$:
\[(\alpha 1 + x)(\beta 1 + y)= (\alpha\beta + f(x,y))1 + (\alpha y +\beta x)\] 	
для $x,y\in V_n,\, \alpha,\beta\in F$. Если форма $f$ невырождена,
то в пространстве $V$ можно выбрать ортонормированный базис.
Если $n$ четно, то в пространстве $V_n$ можно выбрать симплектический базис.

В частности, в $V_2$ можно выбрать симплектический базис $(x,y)$, т.е.
\[f(x,x)=f(y,y)=0,\quad f(x,y)=1\]
(см. \cite{Mal2005}, \cite{Art2016}).
Значит, алгебра $J_2$ может быть задана в виде
\[J_2=\langle 1,x,y\mid x^2=y^2=0, xy=1\rangle.\]
Базис $E=(1,x,y)$ алгебры $J_2$ назовем \emph{каноническим}.

Элемент $a=\lambda 1 + \alpha x + \beta y$
будем записывать в виде $a=(\lambda,\alpha,\beta)_E$;
если ясно в каком базисе задано представление,
то буква $E$ будет опускаться.
\begin{lemma}\label{l:2}
Если элемент $a=(\lambda,\alpha,\beta)$  является $R$-обратимым в $J_2$, то $\lambda \neq 0$.
Далее, элемент $a=(1,\alpha,\beta)$    является   $R$-обратимым тогда и только тогда,
когда $\Delta_a:={\rm det}R_a= 1 -2\alpha \beta \neq 0$.
\end{lemma}
\begin{proof}
Поскольку $x(\alpha x + \beta y)=\beta xy=\beta 1$  и $y(\alpha x + \beta y)=\alpha xy= \alpha 1$,
то элемент $(0,\alpha,\beta)=\alpha x + \beta y$ не является $R$-обратимым.
Далее, если $a=(1,\alpha,\beta)$, то имеем
\begin{equation*}
 \begin{cases}
   1\cdot a = 1\cdot (1+ \alpha x + \beta y) =1+ \alpha x + \beta y, \\
   x\cdot a = x\cdot (1+ \alpha x + \beta y) =x+  \beta 1, \\
   y\cdot a = y\cdot (1+ \alpha x + \beta y) =y+ \alpha 1,
\end{cases}
\end{equation*}
т.е. матрица оператора $R_a$ в базисе $(1,x,y)$ имеет вид
$R_a=\left(\begin{array}{ccc}
   1 & \alpha & \beta \\
   \beta & 1 & 0 \\
   \alpha & 0 & 1
\end{array}\right );$
далее мы не будем различать в обозначениях оператор $R_a$ и его матрицу,
если ясно относительно какого базиса она записана.
Оператор $R_a$  обратим тогда и только тогда, когда ${\rm  det}R_a\neq 0$;
${\rm det}R_a= 1-2\alpha\beta$.
Лемма доказана.
\end{proof}
\begin{remark}
Всякий ниль-элемент индекса 2 алгебры $J_2$ содержится в
множестве $Fx \cup Fy$, в частности, ниль-ранг алгебры $J_2$  равен $2$.
В самом деле, если $a$~--- ниль-элемент индекса 2, то $aR_a=a^2=0$,
т.е. оператор $R_a$ вырожден и $a=(0,\alpha,\beta)$. Тогда имеем:
\[0=a^2=(\alpha x + \beta y)^2 =2\alpha\beta 1,\]
откуда вытекает, что $\alpha=0$ или $\beta=0$, т.е. $a\in Fx \cup Fy$.
\end{remark}

\section{Неунитальная изотопно-простая алгебра $C_2$  ниль-ранга 2}\label{s4}
Рассмотрим коммутативную алгебру $C_2$  и отметим некоторые ее свойства;
эта алгебра имеет \emph{канонический базис} $(a,b,c)$  и следующeе умножение
\[C_2=\langle a,b,c|a^2=b,ab=a,ac=c,b^2=c^2=0,bc=b \rangle.\]

\begin{lemma}\label{l:3}
Алгебра  $C_2$ не имеет единицы $1$.
\end{lemma}
\begin{proof}
В самом деле, пусть $1=\alpha a + \beta b + \gamma c$~---   единица в $C_2$
для некоторых скаляров  $\alpha ,\beta ,\gamma \in F$. Тогда
\[a=a\cdot 1= a(\alpha a + \beta b + \gamma c)=\alpha  a^2 + \beta ab + \gamma ac= \alpha  b + \beta a + \gamma c.\]
Отсюда  $\alpha =0 ,\beta =1 ,\gamma =0$, т.е.  $1=b $ и $b=bc=1c=c$,-- противоречие.
\end{proof}

\begin{lemma}\label{l:4}
Всякий ниль-элемент алгебры $C_2$ содержится в множестве  $Fb \cup Fc$,
значит, ниль-ранг алгебры $C_2$  равен $2$.
\end{lemma}
\begin{proof}
В самом деле, пусть $u=\alpha a + \beta b + \gamma c$~---  ниль-элемент. Тогда
\[0=(\alpha a + \beta b + \gamma c)^2=\alpha^2 b + 2\alpha\beta a + 2\alpha\gamma c + 2\beta\gamma b =\]
\[2\alpha\beta a + (\alpha^2 + 2\beta\gamma) b+ 2\alpha\gamma c,\]
откуда  $\alpha\beta=0, \alpha\gamma=0, \alpha^2 + 2\beta\gamma=0$.
Если  $\alpha\neq 0$, то  $\beta=0$,  $\alpha^2=0$, -- противоречие.
Значит, $\alpha= 0$  и  $u=\beta b + \gamma c$. Поскольку  $\beta\gamma=0$, то  $u\in Fb \cup Fc$, что и требовалось.
\end{proof}
\begin{lemma}\label{l:5}
В алгебре $C_2$  с каноническим базисом $(a,b,c)$  верно:
\[R_a=e_{12}+e_{21}+e_{33};\qquad R_a^2=I=e_{11}+e_{22}+e_{33},\]
где $e_{ij}$~--- матричные единицы и оператор умножения отождествляется с его матрицей 
в каноническом базисе.
\end{lemma}
Указанная лемма немедленно вытекает из определения алгебры  $C_2$.

\begin{lemma}\label{l:6}
Алгебра $C:=C_2$  изотопна йордановой алгебре $J_2$
симметрической билинейной невырожденной формы на $2$-мерном векторном пространстве.
В частности, алгебра $C$ изотопно-проста.
\end{lemma}
\begin{proof}
Рассмотрим алгебру $A=C^{(\varphi,\varphi)}$, $\varphi=R_a$; ее таблица умножения:
\[a\ast a=a^2a^2=b^2=0,\quad b\ast b=(ba)^2=a^2=b,\quad c\ast c=(ca)^2=c^2=0,\]
\[a\ast b=a^2(ba)=ba=a,\quad b\ast c=(ba)(ca)=ac=c,\quad a\ast c=a^2(ca)=bc=b,\]
т.е. $A=\langle a,b,c|a^2=0,b^2=b,c^2=0,ab=a,ac=b,bc=c \rangle$
и после переобозначений $x:=a, 1:=b,y:=c$  получаем алгебру  $J_2$.

Поскольку алгебра $J_2$ унитальна и проста, то она изотопно-проста,
в частности, такой же является и алгебра $C$. Лемма доказана.
\end{proof}

\section{Простые унитальные 3-мерные коммутативные алгебры ниль-ранга 2}
Обозначим через $C(\alpha,\beta,\gamma)$ унитальную коммутативную алгебру вида
\[C(\alpha,\beta,\gamma)=\langle 1,x,y|x^2=y^2=0,xy=\alpha 1 +\beta x + \gamma y\rangle,\]
где $\alpha,\beta,\gamma \in F$; базис $(1,x,y)$  назовем \emph{каноническим}.
\begin{lemma}\label{l:7}
Пусть $C$~--- унитальная коммутативная алгебра, $x,y$~--- ее линейно независимые ниль-элементы индекса 2,
$V=\langle x,y \rangle$~--- подпространство, порожденное элементами $x,y$. Тогда $1\notin V$, т.е. $(1,x,y)$~--- базис алгебры $C$.
Следовательно, алгебра $C$ совпадает с алгеброй $C(\alpha,\beta,\gamma)$ при подходящих скалярах $\alpha,\beta,\gamma$.
\end{lemma}
\begin{proof}
В самом деле, если  $1=\alpha x + \beta y$, то $x=\beta xy$   и  $y=\alpha xy$.
Значит, элементы  $x,y$ линейно зависимы, -- противоречие.
\end{proof}

\begin{lemma}\label{l:8}
Если алгебра $C(\alpha,\beta,\gamma)$ является простой, то  $\alpha \neq 0$.
\end{lemma}
\begin{proof}
Если  $\alpha = 0$, то $V=\langle x,y \rangle$  -- собственный идеал.
\end{proof}
Заметим, что алгебра  $C(1,0,0)$ является йордановой алгеброй симметрической билинейной невырожденной формы $J_2$.

Положим $C(\alpha):=C(\alpha,\alpha,\alpha)$.
\begin{lemma}\label{l:9}
Простая алгебра $C(\alpha,\beta,\gamma)$ изоморфна одной из алгебр:
\begin{center}
$C(\alpha),\, \alpha\neq 0,$  либо  $C(1,1,0)$, либо  $C(1,0,0)$.
\end{center}
\end{lemma}

\begin{proof}
Представим в виде последовательности пунктов.

$1^0$. Докажем, что алгебра $C(\alpha,\beta,\gamma)$ при $\alpha \beta\gamma\neq 0$
изоморфна алгебре  $C(\rho)$, где $\rho=\alpha^{-1}\beta\gamma$.

Заменим базис в  $V=\langle x,y \rangle$:  $x'=\lambda x, y'=\mu y$. Тогда
\[x'y'= (\lambda x)(\mu y)= \lambda \mu (\alpha 1 +\beta \lambda^{-1} x' + \gamma \mu^{-1}y'=
\alpha \lambda \mu 1 + \beta \mu x' + \gamma \lambda y'.\]
Если $\beta\gamma\neq 0$, то полагая  $\lambda = \alpha^{-1}\beta , \mu =\alpha^{-1}\gamma$, получаем
\[\alpha \lambda \mu  =\alpha \alpha^{-1}\beta \alpha^{-1}\gamma = \alpha^{-1}\beta\gamma,\quad
\beta \mu = \beta \alpha^{-1}\gamma = \alpha^{-1}\beta\gamma,\]
\[\gamma \lambda = \gamma \alpha^{-1}\beta = \alpha^{-1}\beta\gamma.\]
Это означает, что алгебра $C(\alpha,\beta,\gamma)$ при $\alpha \beta\gamma\neq 0$
изоморфна алгебре  $C(\rho)$, где  $\rho=\alpha^{-1}\beta\gamma$.

$2^0$. Алгебры $C(\alpha,\beta,\gamma)$ и $C(\alpha,\gamma,\beta)$ изоморфны,
поскольку перестановка $(x,y)$ индуцирует изоморфизм указанных алгебр.

$3^0$. Алгебры $C(\alpha,\beta,0),\,\alpha\beta \neq 0$ и $C(1,1,0)$ изоморфны.

Положим $x'=\alpha^{-1}\beta x, y'=\beta^{-1}y$;
$x', y'$  являются ниль-элементами индекса 2 и
\[x'y'= \alpha^{-1}\beta x\cdot \beta^{-1}y=\alpha^{-1} xy = \alpha^{-1}(\alpha 1 +\beta x)= 1+\alpha^{-1}\beta x=1+x',\]
т.е. алгебры $C(\alpha,\beta,0),\,\alpha\beta \neq 0$, и $C(1,1,0)$ изоморфны.

$4^0$. Алгебры  $C(\omega^2,0,0),\,\omega \neq 0$ и $C(1,0,0)$ изоморфны,
поскольку элементы $x'=\omega^{-1} x, y'=\omega^{-1} y$   являются
ниль-элементами индекса $2$ в алгебре $C(\omega^2,0,0),\,\omega \neq 0$  и выполнены равенства:
$x'y'=(\omega^{-1} x)(\omega^{-1} y)=\omega^{-2} xy=\omega^{-2}\omega^{2}1=1$.
В силу алгебраической замкнутости поля $F$ простые алгебры $C(\alpha,0,0)$ и $C(1,0,0)$ изоморфны.
\end{proof}
\begin{lemma}\label{l:ind2}
Ниль-ранг простой алгебры $C(\alpha)$ равен 2, если $\alpha\neq -2$.
Ниль-ранг алгебры $C(-2)$ равен 3, т.е. указанная алгебра имеет ниль-базис.
\end{lemma}
\begin{proof}
  Допустим, что элемент $a = 1+\lambda x + \mu y$ является ниль-элементом индекса 2
  в алгебре $C(\alpha)$. Тогда
  \[0=a^2= (1+\lambda x + \mu y)^2=1+2\lambda x + 2\mu y +2\alpha\lambda\mu (1+x+y).\]
  После приведения подобных членов получаем систему уравнений
  \begin{equation*}
 \begin{cases}
  1:\, 1+ 2\alpha \lambda\mu=0, \\
  x:\, 2\lambda + 2\alpha \lambda\mu=0, \\
  y:\, 2\mu +  2\alpha \lambda\mu=0,
\end{cases}
\end{equation*}
решением которой является набор $ \lambda=\mu=\frac{1}{2}, \alpha=-2$.
Если же $\alpha\neq -2$, то указанная система решений не имеет.
\end{proof}
Следуя \cite{KP2020}, напомним
\begin{definition}
  Изотоп $C^{(\varphi,\psi)}$ называется \emph{стандартным},
  если операторы $\varphi$ и $\psi$ пропорциональны.
\end{definition}
\begin{lemma}\label{l:inv_nr}
Произвольная алгебра $C$ и её стандартный изотоп $C^{\ast}$
имеют одинаковые ниль-ранги.
\end{lemma}
\begin{proof}
В силу леммы \ref{l:1} можно считать, что стандартный изотоп имеет вид $C^{\ast}=C^{(\varphi,\varphi)}$.
Далее, если $e_1,...,e_k$ линейно независимые ниль-элементы индекса 2 в алгебре $C$,
то $e_1^{\varphi^{-1}},...,e_k^{\varphi^{-1}}$
линейно независимые ниль-элементы индекса 2 в изотопе $C^{\ast}$.
Значит, при стандартной изотопии ниль-ранг не уменьшается.
Поскольку понятие стандартной изотопии симметрично, то ниль-ранги алгебр $C$ и $C^{\ast}$ равны.
\end{proof}

\begin{lemma}\label{l:10}
Алгебры $C(1,1,0)$ и $C(1,0,0)$  изотопны, но не изоморфны.
\end{lemma}
\begin{proof}
Заметим, что вторая алгебра $C(1,0,0)=J_2$ йорданова, а первая нет.
В самом деле, предположив, что алгебра $C(1,1,0)$ йорданова, имеем:
\[2(xy,x,y)=-(y^2,x,x)=0,\]
поскольку $y^2=0$. Однако, в алгебре $C(1,1,0)$ выполнены равенства:
\[(xy,x,y)=(1+x,x,y)=(x,x,y)=-x(xy)=-x(1+x)=-x\neq 0.\]
Значит, указанные алгебры неизоморфны.

Пусть $C:=J_2$ и $(1,x,y)$~--- её канонический базис.
Вычислим оператор правого умножения $R_{c}$ для $c:=1+x$:
\[1c=1+x,\quad xc=x(1+x)=x,\quad yc=y(1+x)=1 + y,\]
$R_{c}=\left(\begin{array}{ccc}
   1 & 1 & 0 \\
   0 & 1 & 0 \\
   1 & 0 & 1
\end{array}\right )$.
Положим
$\varphi = R_{c}^{-1}= \left(\begin{array}{ccc}
   1 & 1 & 0 \\
   0 & 1 & 0 \\
   1 & 0 & 1
\end{array}\right )^{-1}=\left(\begin{array}{ccc}
   1 & -1 & 0 \\
   0 & 1 & 0 \\
   -1 & 1 & 1
\end{array}\right )$.
 Рассмотрим изотоп $C^{\ast}=C^{(\varphi,\varphi)}$ , считая $u\ast v=u^\varphi v^\varphi$ , и в нем элементы:
 \[c:=1+x,\quad a:=x,\quad b:=1+y,\quad e:=c^2=(1+x)^2=1+2x .\]
Тогда
\[a^\varphi=x^\varphi =(0,1,0)\left(\begin{array}{ccc}
   1 & -1 & 0 \\
   0 & 1 & 0 \\
   -1 & 1 & 1
\end{array}\right )= (0,1,0)=x,\]
\[b^\varphi=(1+y)^\varphi =(1,0,1)\left(\begin{array}{ccc}
   1 & -1 & 0 \\
   0 & 1 & 0 \\
   -1 & 1 & 1
\end{array}\right )= (0,0,1)=y,\]
\[e^\varphi=(1+2x)^\varphi =(1,2,0)\left(\begin{array}{ccc}
   1 & -1 & 0 \\
   0 & 1 & 0 \\
   -1 & 1 & 1
\end{array}\right )= (1,1,0)=1+x=c.\]
Заметим, что $1=e-2x=e-2a$. Покажем, что $e$~--- единица изотопа, поскольку для любого элемента $v$:
\[e\ast v = e^\varphi v^\varphi = (c^2R_c^{-1})(vR_c^{-1})= c(vR_c^{-1})=(vR_c^{-1})R_c=v.\]
Далее,
\[a\ast a = (a^\varphi)^2 = x^2= 0,\qquad b\ast b = (b^\varphi)^2 =y^2 = 0,\]
\[a\ast b = a^\varphi b^\varphi=xy= 1= e-2a.\]
Полагая $a'=-2a, b'=-\frac{1}{2}b$, получаем $a'\ast b' = a\ast b =e-2a = e+a'$.
Тем самым  доказано, что $(e,a',b')$-- канонический базис алгебры $C(1,1,0)$,
т.е. алгебры $C^{\ast}=C^{(\varphi,\varphi)}$  и $C(1,1,0)$  изоморфны.
Отсюда вытекает, что $C=J_2$  и $C(1,1,0)$ изотопны.
\end{proof}
\begin{remark}
Проверим, что оператор $R_{1+x}^{-1}$ в алгебре $J_2$ не представим в виде $R_{a}$ ни для какого элемента $a\in C$.
Заметим, что в лемме \ref{l:10}  показано
$R_{1+x}^{-1}= \left(\begin{array}{ccc}
   1 & -1 & 0 \\
   0 & 1 & 0 \\
   -1 & 1 & 1
\end{array}\right )$.
    Если $a=\alpha1+\beta x+\gamma y$, где $\alpha,\beta,\gamma \in F$, то справедливы равенства
    \[    x(1+x)=x(\alpha1+\beta  x+\gamma  y)=\alpha x+\gamma 1=\gamma  1+\alpha x,  \]
    \[y(1+x)=y(\alpha1+\beta  x+\gamma  y)=\alpha y+\beta  1=\beta 1+\alpha y.    \]
    Допустим, что $R_{1+x}^{-1}=R_{a}$. Тогда
$x=\alpha1+\beta  x+\gamma  y,\,	1-x=\gamma  1+\alpha x$.
    Отсюда имеем из первого уравнения $\alpha=0$, а из второго: $\alpha=-1$,
    что невозможно.
\end{remark}
\begin{lemma}\label{l:11}
Алгебры $B:=C(\rho),\,\rho\neq 0;-2$ и $C:=J_2$  изотопны.
\end{lemma}
\begin{proof}
Пусть $(1,x,y)$~--- канонический базис алгебры  $C=J_2$.
Пусть  $c:= 1+\gamma x+y$, $\gamma \neq 0; \frac{1}{2}$. Тогда в алгебре $C=J_2$  имеем:
\[1c= 1+\gamma  x+y,\quad xc=1+x,\quad yc=\gamma 1+y,\]
 \[R_{c} =\left(\begin{array}{ccc}
   1 & \gamma & 1 \\
   1 & 1 & 0 \\
   \gamma & 0 & 1
\end{array}\right ),\qquad{\rm det R_{c}}=1-2\gamma\neq 0.\]
Положим $\varphi=R_{c}^{-1} =\delta\left(\begin{array}{ccc}
   1 & -\gamma & -1 \\
   -1 & 1-\gamma & 1 \\
   -\gamma & \gamma^2 & 1-\gamma
\end{array}\right )$, где $\delta:=(1-2\gamma)^{-1}$.

Рассмотрим изотоп  $C^{\ast}=C^{(\varphi,\varphi)}$, считая  $u\ast v=u^\varphi v^\varphi$,
и в нем элементы:
\[e:=c^2,\quad x':=xc=1+x,\quad y':=yc=\gamma 1+y.\]

Элемент $e$ является единицей в изотопе и имеет представление:
\[e:=c^2 = (1+\gamma x+y)^2 = 1+2\gamma+2\gamma x+2y = \]
\[1+2\gamma+2\gamma (x'-1)+2(y'-\gamma1)=(1-2\gamma) 1 +2\gamma x'+2y'.\]
Тогда
\begin{center}
$1=\delta(e-2\gamma x'-2y')$.
\end{center}
Далее,
$x'^\varphi=x,\quad y'^\varphi=y$, и
\[x'\ast x'=(x'^\varphi)^2=x^2=0,\quad y'\ast y'=(y'^\varphi)^2=y^2=0,\]
\[x'\ast y'=x'^\varphi y'^\varphi=xy=1=\delta(e-2\gamma x'-2y').\]
Следовательно, алгебра $C^{\ast}$ изоморфна алгебре $C(\delta,-2\gamma\delta,-2\delta)$.
В силу леммы \ref{l:9}, п. $1^0$, верно $C(\alpha,\beta,\gamma)\cong C(\alpha^{-1}\beta\gamma)$,
значит, $C(\delta,-2\gamma\delta,-2\delta)\cong C(4\gamma\delta)$.
    Докажем, что для любого $\rho \neq -2$ можно подобрать $\gamma$ так, что
    $ 4\gamma\delta=\rho$. Равенство $ 4\gamma(1 - 2\gamma)^{-1}=\rho$ верно при
    $\gamma=\frac{\rho}{2\rho+4}$. Лемма доказана.
\end{proof}
\begin{remark}
Первое ограничение в лемме $\rho\neq 0$ связано с простотой алгебры, 
а второе $\rho\neq -2$~--- с ниль-рангом алгебры.
\end{remark}
Тем самым доказана
\begin{theorem}\label{th:1}
Простая унитальная коммутативная $3$-мерная алгебра ниль-ранга $2$ 
над алгебраически замкнутым полем $F$
характеристики, отличной от $2$,
изотопна йордановой алгебре $J_2$ симметрической билинейной невырожденной формы
на 2-мерном векторном пространстве.
\end{theorem}
Остается открытым следующий вопрос. \emph{Верно ли, что всякая изотопно-простая $3$-мерная
коммутативная алгебра ниль-ранга $2$ без единицы является изотопом алгебры $J_2$}?
Определенную надежду на положительный ответ дает алгебра $C_2$, введенная в п. \ref{s4}.
\section{Циклотомичная алгебра $C_3$}
Обозначим через  $C_3$ \emph{циклотомичную} алгебру:
\[C_3=\langle x,y,z| x^2=y^2=z^2=0, xy=yx=z,yz=zy=x,zx=xz=y\rangle,\]
которая напоминает $3$-мерную алгебру Ли $V_3$  относительно векторного произведения.
Хорошо известно, что алгебры Ли $V_3$  и $sl_2$  над алгебраически замкнутым полем изоморфны
(см., например, \cite{Jac1979}).
В \cite{KP2020}, теорема 5, доказано, что всякая простая коммутативная $3$-мерная алгебра
с ниль-базисом стандартно-изотопна  алгебре $C_3$. В \cite{KP2020} доказано, что алгебра $C_3$ не унитальна,
 но она обладает стандартным унитальным изотопом. Более того, справедлива
\begin{theorem}\label{th:2}
Всякая простая коммутативная алгебра с ниль-базисом над алгебраически замкнутым полем $F$
характеристики, отличной от $2$, обладает стандартным изотопом
изоморфным алгебре $C(-2)$.
\end{theorem}
\begin{proof}
Пусть $C:=C_3$. Рассмотрим в алгебре $C$ элемент $c:=x+y+z$ и вычислим его квадрат:
\[e:=c^2=(x+y+z)^2=2c.\]
Далее, $xc=x(x+y+z)=xy+xz= y+z$, значит,
\begin{center}
$R_{c} =\left(\begin{array}{ccc}
   0 & 1 & 1 \\
   1 & 0 & 1 \\
   1 & 1 & 0
\end{array}\right )$ и
$\:\varphi=R_{c}^{-1}=\frac{1}{2}\left(\begin{array}{ccc}
   -1 & 1 & 1 \\
   1 & -1 & 1 \\
   1 & 1 & -1
\end{array}\right )$.
\end{center}
Поскольку $(1,1,0)\frac{1}{2}\left(\begin{array}{ccc}
   -1 & 1 & 1 \\
   1 & -1 & 1 \\
   1 & 1 & -1
\end{array}\right )=(0,0,1)$, то $(x+y)^\varphi=z$.
Значит, в силу симметрии справедливы равенства:
\[(x+y)^\varphi=z,\quad (y+z)^\varphi=x,\quad (z+x)^\varphi=y.\]
Таким образом, в изотопе  $C^{\ast}=C^{(\varphi,\varphi)}$, где  $u\ast v=u^\varphi v^\varphi$,
 имеется ниль-базис $x'= y+z$, $y'= z+x$, $z'= x+y$, причем $e=x'+y'+z'$~--- единица в $C^{\ast} $.
 Имеем:
 \[x\ast y'= x'^\varphi y'^\varphi=xy=z.\]
 Полагая
$a:=-2x'=-2y-2z,\quad b:=-2y'=-2z-2x$,
заметим, что
\[e+a+b=2x+2y+2z-2y-2z-2z-2x=-2z.\]
Тогда  $a\ast b=(-2x')\ast (-2y')=4x'\ast y'=4z=(-2)(e+a+b)$.
Значит, алгебра $C^{\ast}$  изоморфна алгебре $C(-2)$. Отсюда в силу \cite{KP2020} получаем требуемое.
\end{proof}

\section{Сильно вырожденные алгебры}

\begin{definition}
Пусть $n\geq 2$. Простая $(n+1)$-мерная алгебра называется \emph{сильно вырожденной},
если она содержит $n$-мерную подалгебру с нулевым умножением.
Ниль-ранг такой алгебры равен $n$. В самом деле, если он равен $n+1$, то
квадрат алгебры имеет размерность $\leq n$, что невозможно в силу простоты сильно вырожденной алгебры.
\end{definition}
Пусть $n\geq 2$. Пусть $G_n$~--- коммутативная алгебра с базисом $(x_1,...,x_n,e)$,
в которой $V=\langle x_1,...,x_n\rangle$~--- подалгебра с нулевым умножением и
\[x_1e=e+x_2,x_ie=x_i+x_{i+1}\,(2\leq i\leq n-1),x_ne=x_n+x_{1}, e^2=e.\]
\begin{proposition}\label{pr:1}
Алгебра $G_n$ является простой.
\end{proposition}
\begin{proof}
Пусть $0\neq I\lhd G_n$, $0 \neq a:=\alpha e + \sum_i\beta_i x_i\in I$.
Знак сравнения $\equiv$ в этом доказательстве будем понимать по модулю $I$.

$1^0$. Докажем, что $\alpha =0$. Допустим, от противного, что $\alpha \neq 0$.
Тогда можно считать, что $\alpha =1$.
Для любого $2\leq i\leq n-1$ имеем
\[ 0\equiv ax_1=ex_1=e+x_2, \qquad
  0\equiv (e+x_2)x_i=x_i+x_{i+1}.\]
  Далее, $0\equiv (e+x_2)x_n=x_n+x_{1}$  и
  $0\equiv (e+x_{2})e=e+x_2+x_3\equiv x_3$.
  Значит, для любого $1\leq i\leq n$ верно $x_i\equiv 0$,
  но тогда $e\equiv 0$ и $I= G_n$, -- противоречие.

 Тем самым доказано, что $0 \neq a = \sum_i\beta_i x_i\equiv 0$.

 $2^0$. Докажем, что $\beta_1=0$.
 Если $\beta_1\neq 0$, то $a'= x_1 + \sum_{i\geq 2}\beta_i x_i\equiv 0$.
 Значит, $0 \equiv a'e= e+v$, где $v \in V$, что противоречит п. $1^0$.

Следовательно, $0 \neq a=\sum_{i\geq 2}\beta_i x_i\equiv 0$
Выберем наибольший индекс $k$ такой, что $\beta_k\neq 0$.

$3^0$. Докажем, что $k < n$. Если $k= n$, то $0 \neq a'' =\sum_{i\geq 2}^{n-1}\beta_i x_i + x_n\equiv 0$.
Тогда $0 \equiv a''e =x_1+\sum_{i\geq 2}^{n}\gamma_i x_i$, что противоречит п.$2^0$.

$4^0$. Итак, для выбранного индекса $k < n$ верно $b :=\sum_{i\geq 2}^{k-1}\beta_i x_i + x_k\equiv 0$.
Значит, $0\equiv be :=\sum_{i\geq 2}^{k-1}\gamma_i x_i + x_k+x_{k+1}$, что противоречит максимальности индекса $k$.
\end{proof}
Алгебра $G_n$ не является изотопно-простой; более того, справедлив более общий факт.

\begin{proposition}\label{pr:2}
Не существует изотопно-простых сильно вырожденных коммутативных алгебр.
\end{proposition}
\begin{proof}
Пусть $A=\langle x_1,...,x_n,t\rangle$~--- изотопно-простая сильно вырожденная коммутативная алгебра (унитальность не предполагается);
$V=\langle x_1,...,x_n\rangle$~--- её подалгебра с нулевым умножением.

В силу простоты алгебры $A$ верно $A^2=A$, значит, $A= \langle x_1t,...,x_nt,t^2\rangle$.
Отсюда вытекает невырожденность оператора $R_{t}$.

Рассмотрим изотоп $C=A^{\ast}=A^{(\varphi,\varphi)}$, $\varphi=R_{t}^{-1}$,
относительно операции $u\ast v =a^\varphi b^\varphi$.
Положим $z_i=x_i^{\varphi^{-1}}\, n=\overline{1,n} $. Тогда
$z_i\ast z_j=z_i^{\varphi} z_j^{\varphi}=x_i x_j=0$, значит,
$Z:=V^{\varphi^{-1}}=\langle z_1,...,z_n\rangle$~--- $n$-мерная подалгебра
с нулевым умножением в изотопе $C$.
Далее, для любого $a\in A$ имеем:
\[a\ast (t^2) = (a^{\varphi})\cdot (t^2)^{\varphi}= (a R_{t}^{-1})((t^2) R_{t}^{-1})=(a R_{t}^{-1})t=a,\]
значит, $C$~--- унитальная алгебра с единицей $1:=t^2$. Поскольку $1\notin Z$, то $C=Z+F1$ и
$Z$~--- собственный идеал в простой алгебре $C$.
Полученное противоречие завершает доказательство.
\end{proof}
Хорошо известно \cite{Alb1942}, что простая конечно-мерная алгебра с единицей изотопно-проста.
Для алгебр с идемпотентом этот результат не верен
как показывают предложения \ref{pr:1} и \ref{pr:2}.

Глизбург Вита Иммануиловна,

Московский городской педагогический университет,

e-mail: glizburg@mail.ru
\[\qquad \qquad\]

Пчелинцев Сергей Валентинович,

Финансовый университет при Правительстве РФ, Москва,

e-mail: pchelinzev@mail.ru


\begin{thebibliography}{99}

\bibitem{Alb1942}
A. A. Albert, Non-associative algebras. I. Fundamental concepts and isotopy, Ann. of Math. (2), 43:4 (1942), 685–707.

\bibitem{Bru1944}
R. H. Bruck, Some results in the theory of linear non-associative algebras, Trans. Amer. Math. Soc., 56 (1944), 141–199.

\bibitem{KP2020}
A. A. Krylov, S. V. Pchelintsev, The isotopically simple algebras with a nil-basis, Comm. Algebra, 48:4 (2020), 1697–1712.

\bibitem{Tka2024}
V. G. Tkachev,
Inner isotopes associated with automorphisms of commu- tative associative algebras,
Communications in Mathematics 32 (2024), no. 2, 153–184 DOI: https://doi.org/10.46298/cm.12223

\bibitem{GP2020}
V. I. Glizburg, S. V. Pchelintsev, Isotopes of simple algebras of arbitrary dimension, Asian-Eur. J. Math., 2020, 2050108, 19 pp.

\bibitem{BGP2021}
Л. Р. Борисова, В. И. Глизбург, С. В. Пчелинцев, Односторонние изотопы конечно-мерных алгебр,
ФПМ, 2021, 23(4), стр. 3–16;
англ. пер.: L. R. Borisova, V. I. Glizburg, and S. V. Pchelintsev,
One-sided isotopes of finite-dimensional algebras, Journal of Mathematical Sciences, Vol. 269, No. 4, January, 2023, pp. 443-452.

\bibitem{Pch2020}
С. В. Пчелинцев, Изототы альтернативных алгебр характеристики, отличной от 3, Изв. РАН. Сер. Матем., 84:5 (2020), 197 - 210;
англ. пер.: S.V. Pchelintsev, Isotopes of alternative algebras of characteristic different from 3, Izvestiya: Mathematics 84:5 (2020), 1002–1015.

\bibitem{Pch2021}
С. В. Пчелинцев, Центральные изототы (-1,1)-алгебр, Сиб. Матем. журн., 62:4 (2021), 830 - 844;
англ. пер.: S.V. Pchelintsev, Central isotopes of (-1,1)-algebras, Siberian Mathematical Journal, 2021, Vol. 62, No. 4, pp. 678–690.

\bibitem{Jac1968} N. Jacobson, Structure and Representations of Jordan Algebras, American Mathematical Society
Colloquium Publications, vol. XXXIX, American Mathematical Society, Providence, RI, 1968.

\bibitem{ZSSS} K.A. Zhevlakov, A.M. Slinko, I.P. Shestakov, A.I. Shirshov, Rings that are Nearly Associative,
Nauka, Moscow, 1978; English translation: Academic Press, NY, 1982.

\bibitem{Mal2005} А.И.Мальцев, Основы линейной алгебры, М.: Наука, 2005.

\bibitem{Art2016} E. Artin, Geometric Algebra, Courier Dover Publications, 2016.

\bibitem{Jac1979} N. Jacobson, Lie Algebras, vol. 10, Courier Corporation, 1979.



\end{thebibliography}
\end{document}